\documentclass{article}
\usepackage{verbatim,amssymb,amscd,amsxtra,amsmath,amsthm,amssymb,latexsym}

\usepackage{url}

\title{The modularity conjecture holds for linear idempotent varieties}

\author{Wolfram Bentz \\
\small{Centro de \'Algebra da Universidade de Lisboa} \\
\small{Av. Prof. Gama Pinto, 2} \\
\small{1649-003 Lisboa, Portugal, wfbentz@fc.ul.pt} \\ 
\and  Lu\'\i s Sequeira\\
\small{Centro de \'Algebra da Universidade de Lisboa} \\
\small{Av. Prof. Gama Pinto, 2} \\
\small{ 1649-003 Lisboa, Portugal} \\{\footnotesize \&}\\
\small{Departamento de Matem\'atica}\\
\small{Faculdade de Ci\^encias, Universidade de Lisboa}\\
\small{1749-016 Lisboa, Portugal, lfsequeira@fc.ul.pt}}

\newcommand{\eps}{\epsilon}

\newcommand{\V}{\ensuremath{\mathcal{V}}}

\newcommand{\M}{\ensuremath{\mathcal{M}}}
\newcommand{\N}{\ensuremath{\mathbb{N}}}

\DeclareMathOperator{\LL}{\textbf{L}}

\newcommand{\set}[1]{\{\,#1\,\}}

\renewcommand{\L}{\mathbf{L}}
\theoremstyle{plain}
\newtheorem{theorem}{Theorem}[section]
\newtheorem{lemma}[theorem]{Lemma}

\newtheorem{corollary}[theorem]{Corollary}

\newtheorem{claim}[theorem]{Claim}
\newtheorem{modconj}{Modularity Conjecture}

\theoremstyle{remark}

\newtheorem{example}{Example}

\theoremstyle{definition}
\newtheorem{defn}[theorem]{Definition}

\begin{document}
\maketitle
\begin{abstract}
The ``Modularity Conjecture'' is the assertion that the join of two nonmodular varieties is nonmodular. We establish the veracity of this conjecture for the case of linear idempotent varieties.
We also establish analogous results concerning $n$-permutability for some $n$, and the satisfaction of nontrivial congruence identities. Our theorems require a technical result about the equational theory of linear
varieties, which might be of independent interest.

\vskip 2mm

\noindent\emph{$2010$ Mathematics Subject Classification\/}. 08B10,~08B05,~03C05.

\vskip 2mm


\end{abstract}

    \maketitle

\section{Introduction}

   The lattice $\LL$ of interpretability types was introduced in
    \cite{Ne74} and thoroughly studied in \cite{GT84}.
    Maltsev conditions, which are
    associated with many important properties of varieties --- such as
    permutability or distributivity of congruence lattices --- correspond
    very nicely to filters of $\LL$.

We assume the reader is familiar with the basic notions of
interpretability of varieties and Maltsev conditions. For the
relevant concepts, the reader is referred to \cite{GT84} or \cite{Se01}.

    The question of whether a given Maltsev condition may be implied
    by, or equivalent to,
    the conjunction of two weaker conditions translates directly
    into the question of whether the corresponding Maltsev filter
    fails to be
    prime, or indecomposable.

    It was shown by Garcia and Taylor \cite{GT84} that the filter of congruence
    distributive varieties is the proper intersection of larger filters. A similar result holds for the filter of all varieties with a near-unanimity operation~\cite{Se03}.

    On the other hand, \cite{GT84} contains the following outstanding
    \emph{modularity conjecture}:

    \begin{modconj}
    In $\LL$, the filter $\M$ of congruence modular varieties is
    prime.
    \end{modconj}

    In other words, the modularity conjecture states that the join of two varieties cannot have ``Day terms'' unless they already exist in one of them.

    It was also conjectured in \cite{GT84} that the filter of congruence
    permutable varieties is prime; this was proved  by Steven
    Tschantz \cite{Ts96}.
    This ``permutability conjecture'' might be arguably easier than the modularity conjecture, and yet Tschantz's proof is extremely difficult and remains unpublished.

        Some work on the modularity conjecture was done in \cite{Se01}. The focus there was on the form of Day terms in a potential counterexample,
         and it was shown that if one exists, the required Day terms must be rather involved, thus ruling out the possibility of an ``easy'' counterexample (see \cite{Se06}).

In the present paper we take a different approach, focusing on varieties of a specific form.

    We use the notion of \emph{derivative} introduced by Dent, Kearnes and Szendrei in \cite{DKS12}, and the results obtained there, to prove that the modularity conjecture holds when restricted to linear idempotent varieties (all the relevant definitions appear in the next section).

    The main result of this paper is:

\begin{theorem}\label{thm:modconjlinid}
 Let $\V_{1}$ and $\V_{2}$ be varieties axiomatized by linear idempotent identities. If $\V_{1} \vee \V_{2}$ is congruence modular, then either $\V_{1}$ or $\V_{2}$ is congruence modular.
\end{theorem}

In \cite{DKS12}, the derivative is also used to give a characterization of those linear idempotent varieties which satisfy some nontrivial congruence identity. Using this characterization, here we are able to prove the following:

\begin{theorem}\label{thm:nontrivconjlinid}
 Let $\V_{1}$ and $\V_{2}$ be varieties axiomatized by linear idempotent identities. If $\V_{1} \vee \V_{2}$ satisfies a nontrivial congruence identity, then either $\V_{1}$ or $\V_{2}$ satisfies a nontrivial congruence identity.
\end{theorem}

Ralph Freese \cite{Fr12} created an ordered version of the derivative introduced in \cite{DKS12} and used it to establish, for the property of being $n$-permutable for some $n$, some  results that are analogous to the ones established in \cite{DKS12} for modularity. Here we make use of the results in \cite{Fr12} to also establish the following result:

\begin{theorem}\label{thm:npermconjlinid}
 Let $\V_{1}$ and $\V_{2}$ be varieties axiomatized by linear idempotent identities. If $\V_{1} \vee \V_{2}$ is $n$-permutable for some $n$, then either $\V_{1}$ or $\V_{2}$ is $n$-permutable for some $n$.\end{theorem}

Definitions of the above notions will by given in Section \ref{sec:notation}, while the results using the derivative and order derivative are shown in Sections \ref{sec:derivative} and \ref{sec:orderderivative}, respectively.
Section \ref{sec:ident-lin-varieties}, which is essentially self-contained, gives the proof of a crucial property of linear varieties. This result might be of independent interest for the study of linearly
generated equational systems.

\section{Definitions and notations} \label{sec:notation}

Let $\Sigma$ be a set of identities. We say that $\Sigma$ is \emph{idempotent} if, for every function symbol $F$ appearing in $\Sigma$, it is the case that
$\Sigma \models F(x,\dots,x) \approx x$. We say that $\Sigma$ is \emph{linear} if each term appearing in $\Sigma$ has at most one function symbol. %
\footnote{Linear terms are also called \emph{depth 1} in \cite{Be06}, \cite{Be07} and \emph{simple} in \cite{Ta09}.}

We say that $\Sigma$ is \emph{inconsistent} if it can only be modeled in trivial varieties, i.~e., if  $\Sigma \models x \approx y$.

We'll say that a variety is (linear) idempotent whenever it is axiomatized by a (linear) idempotent set of identities.

\begin{defn}[\cite{DKS12}]
Let $\Sigma$ be an idempotent set of identities and let $F$ be a function symbol occurring in $\Sigma$.
\begin{itemize}
\item[(i)]  We say that $F$ is \emph{weakly independent} of its $i$-the place if
 $\Sigma \models x \approx F(\mathbf{w})$ for a variable $x$ and some sequence of not necessarily distinct variables $\mathbf{w}$, such that $w_{i} \ne x$.
\item[(ii)] We say that  $F$ is \emph{independent} of its $i$-th place if
 $\Sigma \models F(\mathbf{w}) \approx F(\mathbf{w'})$ where $\mathbf{w}$, $\mathbf{w'}$ are two sequences of distinct variables, that are the same except at position $i$.
\end{itemize}
\end{defn}

\begin{example}
 Let $\Sigma=\set{p(x,y,y)\approx x, p(y,y,x) \approx x}$, the set of identities describing a Maltsev term.
 Then by the first identity $p$ is weakly independent of its second and third positions; by the second identity, $p$ is also weakly independent of its first position.
\end{example}

\begin{defn}[\cite{DKS12}]\label{def:derivative}
Let $\Sigma$ be an idempotent set of identities. The \emph{derivative} $\Sigma'$ is defined by augmenting $\Sigma$ with the identities  asserting that $F$ is independent of its $i$-th place, for all $F$ and $i$ such that $F$ is weakly independent of its $i$-th place.
\end{defn}

Hence, the derivative $\Sigma'$ can be seen as the result of strengthening every occurrence of weak independence in $\Sigma$ to independence. When we have a variety $\V$ axiomatized by a set $\Sigma$ of idempotent identities, we will denote by $\V'$ the variety axiomatized by $\Sigma'$.
The derivative can be iterated in the obvious way. The $n$th derivative of $\Sigma$ (or $\V$) will be denoted by $\Sigma^{(n)}$ (or $\V^{(n)}$).

\begin{example}
 Let $\Sigma$ be the same as in Example 1. Then $\Sigma'$ will contain identities which state that $p$ is independent of all its places:
 $$\Sigma' = \Sigma \cup \set{p(u,y,z)\approx p(v,y,z),p(x,u,z) \approx p(x,v,z),p(x,y,u)\approx p(x,y,v)}$$
 It is easy to see that $\Sigma'$ is inconsistent:
 $$\Sigma' \models x \approx p(x,y,y) \approx p(y,y,y) \approx y$$
\end{example}

The notion of derivative was used in \cite{DKS12} to establish very nice results concerning congruence modularity, and the satisfaction of nontrivial congruence identities, which we will state and use in the next section.
Dent, Kearnes and Szendrei also suggested in \cite{DKS12} that alternative notions of ``derivative'' might be developed, which could
be applied in a similar fashion to other Maltsev properties, such as $n$-permutability. Ralph Freese (\cite{Fr12}) did just that, by defining the notion of \emph{order derivative}:

\begin{defn}[\cite{Fr12}]
  Let $\Sigma$ be an idempotent set of equations. The \emph{order derivative} of $\Sigma$, denoted by $\Sigma^{+}$, is the augmentation of $\Sigma$ by additional identities, in the following way: if $\Sigma \models x \approx F(\mathbf{w})$, for a tuple $\mathbf{w}$ of not necessarily distinct variables, and an operation symbol $F$ occurring in $\Sigma$, then $\Sigma^{+}$ will contain all identities of the form
  $$ x \approx F(\mathbf{w'}) $$
 where, for each $i$, $w_{i}'$ is either $x$ or $w_{i}$.
\end{defn}

\begin{example}
Again, let $\Sigma$ be as in Example 1. Then $\Sigma^{+}$ will include identities such as
$x \approx p(x,x,y)$, $x \approx p(x,y,x)$ and $x \approx p(y,x,x)$ (and also identities that state that $F$ is idempotent, but those add no information, as that already held in $\Sigma$).
Clearly, $\Sigma^{+}$ is inconsistent, for
$$\Sigma^{+} \models x \approx p(x,y,y) \approx y$$
\end{example}

When we have a variety $\V$ axiomatized by a set $\Sigma$ of identities, we will let $\V^{+}$ denote the variety axiomatized by $\Sigma^{+}$.
The $n$th order derivative of $\Sigma$ (or of $\V$) will be denoted by $\Sigma^{+(n)}$ (or $\V^{+(n)}$).

\section{Modularity and nontrivial congruence identities} \label{sec:derivative}

\begin{theorem}[\cite{DKS12}, Thm. 3.2]
A variety $\V$ is congruence modular if and only if $\V$ realizes some set $\Sigma$ of idempotent identities such that $\Sigma'$ is inconsistent.
\end{theorem}

It is noted in \cite{DKS12} that, while the fact of $\Sigma'$ being inconsistent forces any variety realizing $\Sigma$ to be congruence modular, the converse is not true in general.
However, for \emph{linear} idempotent varieties, a stronger result could be obtained. The following theorem is a slight reformulation of Theorem 5.1 of \cite{DKS12}:

\begin{theorem}\label{thm:modiffderone}
 Let $\V$ be a linear idempotent variety. Then $\V$ is congruence modular if and only if $\V'$ is a trivial variety.
\end{theorem}

Throughout this and the following section, $\Sigma_{1}$ and $\Sigma_{2}$ will always denote idempotent sets of identities, taken with disjoint sets of function symbols.

 If $\V_{1}$, $\V_{2}$ are the varieties axiomatized by $\Sigma_{1}$, $\Sigma_{2}$, respectively, then it is well known (see \cite{GT84}) that
 $\V_{1} \vee \V_{2}$ is exactly the variety axiomatized by $\Sigma_{1} \cup \Sigma_{2}$.

 Furthermore, it is clear that $(\Sigma_{1} \cup \Sigma_{2})' \supseteq \Sigma_{1}' \cup \Sigma_{2}'$.
 Corollary \ref{cor:deriv-union} below shows that in fact equality holds when $\Sigma_{1}$ and $\Sigma_{2}$ are \emph{linear}. This result will come out as a consequence of Lemma~\ref{lem:weakindep-union}.
This lemma provides the leverage we need to establish all the results stated in the introduction.
While the result stated in the lemma may seem intuitively obvious, its proof involves arguments about term manipulations that are quite technical in nature, and is presented in Section \ref{sec:ident-lin-varieties}.

\begin{lemma}\label{lem:weakindep-union}
  Let $\Sigma_{1}$ and $\Sigma_{2}$ be consistent sets of linear idempotent identities. Let $F$ be an operation symbol occurring in
  $\Sigma_{i}$, where $i \in \set{1,2}$. Consider a linear equation of the form
  $F(\mathbf{x})\approx y$ over a variable set $X$. If $\Sigma_{1} \cup \Sigma_{2} \models F(\mathbf{x})\approx y $, then $\Sigma_{i} \models F(\mathbf{x})\approx y $.
\end{lemma}

\begin{corollary}\label{cor:deriv-union}
 If $\Sigma_{1}$ and $\Sigma_{2}$ are sets of linear idempotent identities, then
 $$(\Sigma_{1} \cup \Sigma_{2})'=\Sigma_{1}' \cup \Sigma_{2}'$$
\end{corollary}
\begin{proof}
Immediate from the previous lemma.
\end{proof}

\begin{corollary}\label{cor:derjoin}
 Let $\V_{1}$, $\V_{2}$ be linear idempotent varieties. Then $(\V_{1} \vee \V_{2})' = \V_{1}' \vee \V_{2}'$.
\end{corollary}

\begin{corollary}\label{cor:derjoinn}
 Let $\V_{1}$, $\V_{2}$ be linear idempotent varieties. Then for every $n \in \N$, $(\V_{1} \vee \V_{2})^{(n)} = \V_{1}^{(n)} \vee \V_{2}^{(n)}$.
\end{corollary}

We also need the following fact from \cite{GT84}:

\begin{lemma}[\cite{GT84}]\label{lem:oneisprime}
 $\mathbf{1}$ is join-prime in $\L$.
\end{lemma}

Now we have all the tools we need to prove our main theorem.

\begin{proof}[Proof of Theorem \ref{thm:modconjlinid}]
Let $\V_{1}$, $\V_{2}$ be linear idempotent varieties.

Suppose that $\V_{1}$ and $\V_{2}$ are not congruence modular.
By Theorem \ref{thm:modiffderone}, $\V_{1}'$ and $\V_{2}'$ are not trivial, i.e., they are different from $\mathbf{1}$ in $\L$.
Let $\V = \V_{1} \vee \V_{2}$.  By Corollary \ref{cor:derjoin}, $\V'=\V_{1}' \vee \V_{2}'$, and Lemma \ref{lem:oneisprime} guarantees that
$\V'$ is not trivial. Since $\V$ is a linear idempotent variety, Theorem \ref{thm:modiffderone} gives us that $\V$ is not congruence modular.
\end{proof}

Theorem~\ref{thm:modiffderone} shows that, for linear idempotent varieties, congruence modularity is equivalent to the derivative being inconsistent.
The following result relates the satisfaction of some nontrivial congruence identity  to the inconsistency of some iteration of the derivative.

\begin{theorem}[\cite{DKS12}]\label{thm:linidnontriv}
 Let $\V$ be a linear idempotent variety. Then $\V$ satisfies some nontrivial congruence identity if and only if for some $n$, $\V^{(n)}$ is trivial.
\end{theorem}

\begin{proof}[Proof of Theorem~\ref{thm:nontrivconjlinid}]
 Let $\V_{1}$, $\V_{2}$ be linear idempotent varieties, and suppose that $\V_{1} \vee \V_{2}$ satisfies some nontrivial congruence identity.
 By Theorem~\ref{thm:linidnontriv}, there is some natural number $n$ such that $(\V_{1} \vee \V_{2})^{(n)}$ is trivial.
 Hence, by Corollary~\ref{cor:derjoinn} and Lemma~\ref{lem:oneisprime}, either $\V_{1}^{(n)}$ or $\V_{1}^{(n)}$ is trivial. Again, by Theorem~\ref{thm:linidnontriv}, we have that either $\V_{1}$ or $\V_{2}$ satisfies a nontrivial congruence identity.
\end{proof}
\section{$n$-permutability}\label{sec:orderderivative}

In this section, we take care of the proof of Theorem~\ref{thm:npermconjlinid}, which will follow along the very same lines as those of
Theorems~\ref{thm:modconjlinid} and \ref{thm:nontrivconjlinid}. The first result is an analog, for the order derivative, of Corollary~\ref{cor:deriv-union}:

\begin{lemma}
  If $\Sigma_{1}$ and $\Sigma_{2}$ are sets of linear idempotent identities, then
 $$(\Sigma_{1} \cup \Sigma_{2})^{+}=\Sigma_{1}^{+} \cup \Sigma_{2}^{+}$$
\end{lemma}
\begin{proof}
 Again, this is an immediate consequence of Lemma~\ref{lem:weakindep-union}.
\end{proof}

\begin{corollary}\label{cor:orderderjoinn}
 Let $\V_{1}$, $\V_{2}$ be linear idempotent varieties. Then for every $n \in \N$, $(\V_{1} \vee \V_{2})^{+(n)} = \V_{1}^{+(n)} \vee \V_{2}^{+(n)}$.
\end{corollary}

The following result is part of Theorem~7 of \cite{Fr12}.

\begin{theorem}
 Let $\Sigma$ be a set of linear idempotent identities. Then the variety axiomatized by $\Sigma$ is $n$-permutable for some $n$ if and only if some iterated order derivative of $\Sigma$ is inconsistent.
\end{theorem}

\begin{proof}[Proof of Theorem~\ref{thm:npermconjlinid}]
The result follows just as in the proof of Theorem~\ref{thm:nontrivconjlinid}, using the order derivative instead of the derivative.
\end{proof}


\section{Identities in linear varieties}\label{sec:ident-lin-varieties}

In this section we prove the important technical Lemma \ref{lem:weakindep-union}, on which the proofs of Theorems 1.1, 1.2 and 1.3 were based. The proof uses arguments over rewriting sequences in a fashion similar to
\cite{Be06} and \cite{Be07}.

In the following we will introduce terminology for this task.
Note that some of our notations are variants (usually generalizations) of established meanings. Also, some concepts have complicated formal descriptions but are easy to grasp informally;
we will occasional stick with an informal term whose meaning should be clear in order to avoid excessive notation.

For a given signature, consider a set $\Sigma$ of linear equalities over a set $V$ and the set $T(X)$ of $\Sigma$-terms over a set $X$, which we may consider to be infinite.
For simplicity we will use $x,y,z$ to refer to elements of
$X$ and $v,w$ to refer to elements of $V$ (occasionally with subscripts).

Given a term $t \in T(X)$, an \emph{occurrence} $s$ in $t$ is a path in the syntactic tree of $t$ together with the subterm of $t$ corresponding to the subtree reached by the path. We will refer to the path as the \emph{position}
of the occurrence and use $\bar s$ to denote the corresponding subterm. Note that a position can be identified with a (potentially empty) finite list of integers. Given an occurrence $s$ in an occurrence $t$ and
another occurrence $s'$ in an occurrence $t'$, we can talk of the occurrence that is in the same relative position towards $s'$ as $s$ is to $t$ by concatenating
 the position of $s'$ in $t'$ with the position of $s$ in $t$; this requires that the resulting position is ``compatible" with $t'$ in an obvious way.

We will extend structural concepts from terms onto occurrences. For example, if we say that an occurrence $t$ has the form
$f(s_1, \dots, s_n)$ for some occurrences $s_i$, this will mean that the terms $\bar s_i$ satisfy $\bar t= f(\bar s_1, \dots,\bar s_n)$ and that the position of $s_i$ is the $i$-th child of $t$. If we introduce an occurrence
$t$ without positional context, its position should be taken as the root of the corresponding term $\bar t$.

A \emph{derivation} in $\Sigma$ is a sequence $t_0, t_1,\dots, t_n$ of occurrences, such  that
$t_{i}$ is obtained from $t_{i-1}$ by one rewriting step with an equality $\eps_i \in \Sigma \cup \Sigma^{\delta}$. Existence of a derivation is clearly equivalent to $\Sigma \models \bar t_0 \approx \bar t_n$.
We additionally require
that a derivation is enriched with  enough ``syntactical information" to completely reconstruct
the rewriting procedure. Concretely, for each pair $(t_{i-1}, t_{i})$ an explicit $\eps_i$  is given together with two occurrences $s$ in $t_{i-1}$ and $s'$ in $t_{i}$ such that $t_i$ is obtained from $t_{i-1}$ by rewriting
$\bar s$ into $\bar s'$ using the equality $\eps_i$. We will moreover adopt the convention that if $t \approx t'$ is a rewriting step using the equality $u \approx u'$, then  $u$ matches up with
 $\bar s$ and  $u'$
with $\bar s'$. We will also consider reverse  pairs $(t_{i}, t_{i-1})$ to be single rewriting steps of the derivation, in this case we use $\eps_i^{\delta}$ as the corresponding equality.

\begin{example} \label{ex:deriv}
In the following examples, the relevant equality is applied to the occurrences specified by underlining.
\begin{enumerate}
\item Let $t=f(\underline{g(h(x),h(x),x)})$, $t'=f(\underline{h(x)})$, then $t'$ is obtained from $t$ by rewriting with the equation $g(v,v,w) \approx v$. Note that
$t'$ could have also been obtained thought the equation $g(v_1,v_2,v_3) \approx v_1$, and that $t$ is obtained from $t'$ by rewriting with $v \approx g(v,v,w)$.
\item $t'= \underline{f(y,x)}$ is obtained from $ \underline{f(x,y)}$ by the equation $f(u,v) \approx f(v,u)$. It could have also been obtained by the equation $f(v,u) \approx f(u,v)$.
\item $t'=f(\underline{f(f(x))})$ is obtained from $f(\underline{f(x)})$ with the equation $f(v) \approx f(f(v))$, using the underlined occurrences. It could have also been obtained with the same equation applied to the
 occurrences
given by  $t=\underline{f(f(x))}$, $t'=  \underline{f(f(f(x)))}$.
\end{enumerate}
\end{example}

The syntactic information contained with each derivation will ensure that various constructions below will be well-defined.
Now consider a fixed derivation $t=t_0, t_1,\dots, t_n=t'$. By an \emph{occurrence of the derivation}, we mean an index $i \in \{0,\dots,n\}$ together with an occurrence $s$ of $t_i$.
Once again, let $\bar{s} \in T(X)$ denote the underlying term of $s$.
The occurrences of a derivation are naturally ordered by inclusion; we will denote this order by ``$\le$". Note that occurrences from different $t_i$ are always incomparable under $\le$.

We will next define a quasi-order $\succeq$ on the occurrences of a derivation in terms of a generating relation $\succeq'$. Let $u \le t_i=:t$ and let $t'$ be one of $t_{i-1}, t_{i+1}$, such that
$t'$ is obtained from $t$ by rewriting with the equation $\eps$. Let $s \le t, s' \le t'$ be the occurrences involved in the rewriting step.
We will define pairs of the form $(u,\underline{\phantom{u}})$ in $\succeq'$ according to particulars of the rewriting step in relation to $u$, by distinguishing several cases. For our definition to be well-defined, we need
that all positions are actually valid in their containing occurrence; the routine verification of this is left to the reader.
\begin{defn} \label{def:quasi}
With notation as above, we define pairs in $\succeq'$ as follows:
\begin{enumerate}
\item If $u \not\le s$ and $ s \not\le u$, let $u' \le t'$ be the occurrence  that is in the same position in $t'$ as $u$ is in $t$. In this case we
set $u\succeq'u'$. Note that $\bar u = \bar u'$. \label{c:incomp}
\item If $s<u$, let $u' \le t'$ be the occurrence that is in the same  position in $t'$ as $u$ is in $t$;
set $u\succeq'u'$. Note that $\bar u'$ can be obtained from $\bar u$ by rewriting with $\eps$.  \label{c:below}
\item  \label{c:above} This case covers the situation that either $u<s$ (for arbitrary $\eps$) or that $u=s$ and $\eps$ is of the form $v\approx f(v_1,\dots, v_n)$ or $v \approx v$ for some $v,v_i \in V$. As $\eps$ is linear,
 there is a unique $w \in V$ appearing in the left hand side
 of $\eps$ and a
unique occurrence $p$, $u \le p \le s$, such that the rewriting step matches  $p$ with $w$ ($p$ is either equal to $s$, if $\eps$ has the form $v\approx f(v_1,\dots, v_n)$, in which
case $v=w$, or is ``one level below" $s$, in the other cases).

For each occurrence $p^\ast_i$ with $\bar p^\ast_i = w$ in $\eps$, let  $p_i$ be the occurrence in the same position  in $s$ or $s'$, as $p_i^\ast$ is in the left or right hand side of $\eps$, respectively.
All of these occurrences have the same underlying term as $p$ (note that $p$ itself is one of the $p_i$; it is possible that $p$ is the only such occurrence). Let $u_i$ be the occurrences that are in the
same position in $p_i$ as $u$ is in $p$, all of which have the same underlying term as $u$. We set $u\succeq'u_i$ for each such $u_i$.

\item \label{c:equal} Let $u=s$ and $\eps$ have the form  $f(v_1,\dots,v_i)=g(w_1,\dots,w_j)$ or   $f(v_1,\dots,v_i)=v$. We set $u \succeq' s'$.
\end{enumerate}
We define $\succeq'$ to be the smallest set obtained by the above rules, as $(t,t')$ run through all pairs of the form $(t_i,t_{i-1})$ and $(t_i, t_{i+1})$, and $u$ runs though all occurrences with $u \le t_i$.
We let $\succeq$ be the reflexive and transitive closure of $\succeq'$.
\end{defn}
\begin{example}
Below, we will give various examples of occurrences $t$, rewriting equation $\eps$, and rewritten occurrence $t'$. As in Example \ref{ex:deriv}, we will indicate syntactical information by underlining.
An overbrace will indicate the occurrence corresponding to $u$ while an underbrace will indicate all occurrences $u'$ with
$u \succeq' u'$. Our numbering corresponds to that in Definition \ref{def:quasi}.
\begin{enumerate}
 \item $f(\overbrace{g(x)},\underline{f(y)})$; $f(v) \approx v$; $f(\underbrace{g(x)}, \underline{y})$
 \item $f(\overbrace{g(x,\underline{f(y)})})$;  $f(v) \approx v$; $f(\underbrace{g(x,\underline{y})})$
 \item \begin{enumerate} \item $f(\underline{g(\overbrace{\underbrace{h(x,y)}},\underbrace{h(x,y)},h(x,y))})$; $g(v,v,w) \approx h(v,w)$; $f(\underline{h(\underbrace{h(x,y)},h(x,y))})$
 \item $f(\underbrace{\overbrace{\underline{g(x,y)}}})$; $v \approx h(v,v,w)$; $f(\underline{h(\underbrace{g(x,y)},\underbrace{g(x,y)},g(x,y))})$
 \end{enumerate}
\item $f(\underline{\overbrace{g({h(x,y)},{h(x,y)},h(x,y))}})$; $g(v,v,w) \approx h(v,w)$; $f(\underbrace{\underline{h(h(x,y),h(x,y))}})$
 \end{enumerate}
\end{example}

 The following lemma is obvious from the definition of $\succeq$.
\begin{lemma} \label{lem:Sequivalent}
With notation as above, if $s \succeq s'$ then $\Sigma \models \bar s \approx \bar s'$.
\end{lemma}

We now use the above definitions in the situation of Lemma \ref{lem:weakindep-union} with $\Sigma:= \Sigma_1 \cup \Sigma_2$.
 Without loss of generality, assume that  $F$ occurs in $\Sigma_{1}$  and that
 there is an identity
\begin{align*}
F(\mathbf{x}) \approx y  \tag{$Id$}
\end{align*}
such that  $\Sigma_{1} \cup \Sigma_{2} \models (Id)$.
Fix a derivation of this identity,  $t_{0},\dots,t_{n}$, with $F(\mathbf{x})= \bar t_{0}$, $\bar t_{n}=y$, corresponding equalities $\eps_i \in \Sigma \cup \Sigma^{\delta}$, and related occurrences $s_i \le t_{i-1} $, $s_i' \le t_i$.

\begin{claim} \label{cl:variable}
With notation as above, we have that $t_0 \succeq t_n$.
\end{claim}

Let $T$ be the set of all occurrences $u$ in the derivation such that $t_0 \succeq u$. Pick a variable $z \in X$ that does not occur in the derivation. For each $t_i$
 we will define an occurrence $p_i$ as follows. Consider each occurrence such that $u \le t$ for some $t \in T$ and that the underlying term of $u$ is a variable from the set $X$. Now let $\bar p_i$ be the term that it obtained
 from $t_i$ by replacing the variable of each such $u$ with $z$.

 \begin{lemma}\label{lem:2nd deriv}
  $p_0,\dots, p_n$ is once again a $\Sigma$-valid derivation, with the same $\eps_i$ and corresponding syntactical information.
  \end{lemma}
  \begin{proof}
 Consider the step  from $t_{i-1}$ to $t_i$, so that $\eps:=\eps_i$ rewrites the occurrence $s:=s_i \le t_{i-1}$ into  $s':=s_i' \le t_i$.
 Let $q$ denote the occurrence that is in the same position in $p_{i-1}$ as $s$ is in $t_{i-1}$, and  let $q'$ denote the occurrence that is in the same position in $p_{i}$ as $s'$ is in $t_{i}$.

 We first show that $\bar q$ is an instance of the left hand side of $\eps$.
  This is trivial of the left hand side of $\eps$ is a variable. So assume that the left hand side of $\eps$ has the form $f(\mathbf{v})$.
 Let $v_j=v_k$ for some $j \ne k$. As $\bar s$ is an instance of $f(\mathbf{v})$, we have that $ s=f(\mathbf{d})$, for some occurrences satisfying $\bar d_j=\bar d_k$. As the definition of $p_{i-1}$ only affects variables,
 $ q =f(\mathbf{e})$, where each $e_l$ can be obtained from the corresponding $d_l$ by changing some variables to $z$. In order to show that $q$ is an instance of  $f(\mathbf{v})$, we have to show that
 this happens for the same occurrences of variables in both $d_j$ and $d_k$. We will go through the various cases in which a variable can switch.

 If some occurrence $o \ge s$ is in $T$, then all variables in $e_j$ and $e_k$ are $z$ and we have $\bar e_j= \bar e_k$, as needed. If any occurrence $o_j \le d_j$ is in $T$,
 then the occurrence $o_k \le d_k$ that is in the same position in $d_k$ as $o_j$ is in $d_j$
  belongs to $T$ as well, by Case \ref{c:below}, and vice versa, by the symmetry in Case \ref{c:below}.
  Hence once again $\bar e_j= \bar e_k$ and we can conclude that $q$ is an instance of $f(\mathbf{v})$.

 It remains to show that the result of rewriting $p_{i-1}$ with $\eps$ applied to occurrence $q$ is $p_i$. It once again suffices to consider occurrences whose underlying  terms  are variables.
So let $u \le t_{i-1}$ be such an occurrence. Suppose that $u \not \le s$. Then there is a $u'$ in the corresponding position in $t_i$ with $\bar u=\bar u'$. Now suppose
that there is some occurrence $o \in T$ with $o \ge p$. Then $o$ is either incomparable with $s$ or $o > s$. %
In either case there is an occurrence $o'\in T$ in the same position in $t_i$ as $o$ is in $t_{i-1}$,
by Case \ref{c:incomp} or Case \ref{c:above} of the definition of $\succeq'$. The same argument holds in reverse if we start with the assumption that $u'$ is contained in an occurrence $o' \in T$
(recall that our definition of $\succeq'$ also included the reverse rewriting operations from $t_i$ to $t_{i-1}$). Hence $u$ is contained in an occurrence from $T$ if and only if $u'$ is.
Hence, when transforming from $(t_{i-1},t_i)$ to $(p_{i-1},p_i)$, $u$ changes its term to $z$ if and only if $u'$ changes. It follows that $p_{i-1}$ and $p_{i}$ are identical outside
 of the rewriting effected
occurrences $q$ and $q'$, as required.

Now let $v$ be a variable appearing in the equality $\eps$. Pick any occurrences $a \ne a'$ that correspond to $v$ during the rewriting process. Then $\bar a =\bar a'$.
In the case that $a,a' \le s$, we have already seen that variables switch to $z$ in exactly the same positions in $a$ and $a'$ when showing that $q$ is an instance of  the left hand side of $\eps$, and the same
result follows by symmetry if $a,a' \le s'$. Hence assume w.l.o.g. that $a\le s$, $a' \le s'$. %

Let $u \le a$ with
$\bar u$ being a variable and $u'$ be the occurrence that is in the same position in $a'$ as $u$ is in $a$. Now assume that there is an occurrence $o \in T$ with $o \ge u$.
   If $o \ge s$, then by Case \ref{c:below}, $s'$ and hence $a'$ and $u'$ is also contained in an occurrence $o' \in T$. If $o = s$ and $\eps$ has a form as in Case \ref{c:equal}, then by that case $s' \in T$, with
   $u' \le s'$. If $o =s$ and $\eps$ has the form ``$v \approx f(\mathbf{w})$" or ``$v \approx v$",
   then $a=s=o \in T$ and hence $u' \le a' \in T$ by Case \ref{c:above}.
   If $u \le o < s$, then $o \le a $, and, also by Case \ref{c:above},
   there is an $o' \in T$ with $u' \le o' \le a'$, namely the one that is in the same relative position in $a'$ as $o$ is in $a$.
In all cases, if $u$ is contained in an occurrence from $T$ than so is $u'$, and by symmetry, the converse holds.
As above $\bar b =\bar b'$ follows, where $b, b'$ are the occurrences in $p_{i-1},p_i$
that correspond to $a,a'$, and hence are also matched to the same variable $v$ in $\eps$.
Hence the underlying terms of any two occurrences matched to same variable in $\eps$ agree. As $p_{i-1}$ is unchanged from $p_i$ outside of the rewriting area, it follows that
$p_i$ can be obtained from $p_{i-1}$ by rewriting with $\eps$.
Hence $p_0,\dots, p_n$ is a $\Sigma$-valid derivation.
\end{proof}

We are now able to prove Claim \ref{cl:variable}. By Lemma \ref{lem:2nd deriv}, $\Sigma \models \bar p_0 \approx \bar p_n$. As $p_0 \in T$ by definition, we have $\bar p_0 = F(z,\dots,z)$.
If $p_0 \not\succeq p_n$, i.e. $p_n \not\in T$
we would have  $\Sigma \models F(z,\dots,z) \approx y$ with $z \ne y$, which implies inconsistence. The claim follows by contradiction from Lemma \ref{lem:oneisprime}.

By the definition of $T$ and $\succeq$ there are occurrences $q_0 \succeq' q_1 \succeq' \dots \succeq' q_k$ with $q_0 = t_0$ and $\bar q_k=\bar t_n$. By potentially shortening the
sequence, we may assume that $\bar q_i \ne y$ for $i<k$.
It is easy to see by induction that for all $q_i \ne r_k$, $\bar q_i$ has the form $f_i(\dots)$, where all $f_i$ are function symbols from the signature of $\Sigma_1$.

Now define a sequence of occurrences $r_0, r_1, \dots, r_k$ by setting $\bar r_i= \bar q_i$ for all $i$.

\begin{lemma} \label{lem:SSeq}
The sequence $r_0, r_1, \dots, r_k$ can be extended to a derivation over $\Sigma'=\Sigma \cup \{v\approx v\}$ by suitable  syntactical information. Moreover, this can be done while avoiding any rewriting step
using an equation
of the form $v \approx f(\mathbf{w})$ in a way in which the left hand side variable $v$ is matched to any $r_i$.
\end{lemma}

\begin{proof}
Let  $0 \le i <k$. If $r_i \succeq' r_{i+1}$ follows from the rules of either Case \ref{c:incomp} or Case \ref{c:above}, then $\bar r_i =\bar r_{i+1}$, and this is a rewriting step with the trivial equation $v \approx v$.

Assume next that $r_i \succeq' r_{i+1}$ follows by   Case \ref{c:below} or Case \ref{c:equal}. Then
$r_i  \le t$, $r_{i+1} \le t'$ where
$t=t_j$ for some $j$ and $t'$ is either $t_{j-1}$ or $t_{j+1}$. Let $t'$ be obtained from $t$ by rewriting
$s\le t$ into $s'\le t'$ using $\eps$.

In Case \ref{c:equal}, $\eps$ has the form $f(v_1,\dots,v_i) \approx g(w_1,\dots,w_j)$ or   $f(v_1,\dots,v_i) \approx w$,
 $s=r_i$ and $s'=r_{i+1}$ and hence $\bar r_{i+1}$ is obtained from $\bar r_{i}$ by rewriting with  $\eps$, applied to the entire term, as required.

Finally, let the situation be as in Case \ref{c:below}. Then $s < r_i$, $s' < r_{i+1}$, and hence
$\bar r_{i+1}$ is obtained from $\bar r_i$ by rewriting a strict smaller occurrence than $r_i$ with $\eps$.

This proves the first assertion. Now, in Cases \ref{c:incomp}, \ref{c:above}, and \ref{c:equal}, the equality connecting
$r_i$ and $ r_{i+1}$ does not have the form $v \approx f(\mathbf{w})$, while in Case \ref{c:below},
the equality is applied to a strict smaller occurrence than  $r_i$. The second assertion follows.
\end{proof}

Now let $R$ be the collection of all occurrences appearing in the derivation $r_0, \dots r_k$. Choose a function $\varphi$ from $R$ to $X$ satisfying
\begin{enumerate}
\item $\varphi(r)=\varphi(r')$ if and only if $\Sigma \models \bar r \approx \bar r'$.
\item $\Sigma \models \bar r \approx x$ implies $\varphi(r)=x$ for all $x \in X$.
\end{enumerate}
As $X$ is infinite such a  function exists.
We now define another sequence of occurrences. For $0 \le i <k $, if $r_i$ has the form $f(d_{1}, \dots, d_{n})$ set $u_i := f(\varphi(d_1), \dots, \varphi(d_n))$. In addition,
set $u_k=r_k$.
\begin{lemma}
The sequence $u_1, \dots, u_k$ can be extended to a derivation using only equations from $\Sigma_1 \cup \{v \approx v\}$.
\end{lemma}
\begin{proof}
For $0 \le i < k$, consider the rewriting step from $r_i$ to $r_{i+1}$ by the equation $\eps$. If this rewriting step involves rewriting a strictly smaller occurrence than $r_i$, or if it is trivial,
then $\bar u_i= \bar u_{i+1}$ by our definition of $\varphi$ and so $u_{i+1}$ follows from $u_i$ by the equation $v \approx v$. Otherwise, $\bar r_i$ must be an instance
 of the left hand side of $\eps$. By Lemma \ref{lem:SSeq}
$\eps$ cannot have the form $v \approx f(\mathbf{w})$, and hence must look like $f(v_1,\dots,v_i) \approx g(w_1,\dots,w_j)$ or   $f(v_1,\dots,v_i) \approx w$. As mentioned above,
$r_i$, and hence $u_i$ has the form $f_i(\dots)$, with $f_i$ being from the signature of $\Sigma_1$. It follows that $\eps \in \Sigma_1$. Moreover, as $\eps$ is linear,
$\bar u_i$ is an instance of $\eps$ and is rewritten into $\bar u_{i+1}$ by $\eps$, once again by our definition of $\varphi$. The result follows.
\end{proof}

We are now ready to show Lemma \ref{lem:weakindep-union}.
\begin{proof}
The previous lemma shows that $\Sigma_1 \models \bar u_0 \approx \bar u_k$.
But $F(\mathbf{x}) = \bar t_0= \bar r_0= \bar u_0$, where the last equality follows from the second property of $\varphi$. In addition $\bar u_k=\bar r_k =y$ and hence
$\Sigma_1 \models F(\mathbf{x})\approx y$. The result follows.
\end{proof}

\section*{Concluding remarks}

Our results provide further evidence in support of the modularity conjecture.
Theorem \ref{thm:modconjlinid} shows that the conjecture holds in the special case of linear idempotent varieties.
Our methods rely heavily on the properties of linear varieties, indicating that the general hypothesis is unlikely to be solved using this line of inquiry.

Our results about linear varieties, in particular Lemma \ref{lem:weakindep-union}, could prove useful in studying other properties of linear systems, potentially far removed from the topics of this paper.

\section* {Acknowledgment}
The first author  has received funding from the
European Union Seventh Framework Programme (FP7/2007-2013) under
grant agreement no.\ PCOFUND-GA-2009-246542 and from the Foundation for
Science and Technology of Portugal.

 The second author acknowledges support from FCT under Project PEst-OE/ MAT/UI0143/2011.

\end{document}